\documentclass[12pt]{amsart}
\usepackage{amssymb}
\newtheorem{Def}{Definition}
\newtheorem{Lem}{Lemma}

\newtheorem{Thm}{Theorem}
\newtheorem{Cor}{Corollary}
\newtheorem{Rem}{Remark}
\newenvironment{Pf}{ Proof.}{\(\square\)}

\title[On compatible linear connections with totally anti-symmetric torsion tensor...]{On compatible linear connections with totally anti-symmetric torsion tensor of three-dimensional generalized Berwald manifolds}
\author{Cs. Vincze}
\address{Inst. of Math., Univ. of Debrecen \\
H-4010 Debrecen, P.O.Box 12 \\
Hungary}
\email{csvincze@science.unideb.hu}
\keywords{Finsler spaces, Generalized Berwalds spaces, Intrinsic Geometry}
\subjclass{53C60, 58B20}
\begin{document}
\begin{abstract}
Generalized Berwald manifolds are Finsler manifolds admitting linear connections such that the parallel transports preserve the Finslerian length of tangent vectors. By the fundamental result of the theory \cite{V5} such a linear connection must be metrical with respect to the averaged Riemannian metric given by integration of the Riemann-Finsler metric on the indicatrix hypersurfaces. Therefore the linear connection is uniquely determined by its torsion tensor. If the torsion is zero then we have a classical Berwald manifolds. Otherwise the torsion is a strange data we need to express in terms of quantities of the Finsler manifold. In the paper we are going to give explicit formulas for the linear connections with totally anti-symmetric torsion tensor of three-dimensional generalized Berwald manifolds. The results are based on averaging of (intrinsic) Finslerian quantities by integration over the indicatrix surfaces. They imply some consequences for the base manifold as a Riemannian space with respect to the averaged Riemannian metric. The possible cases are Riemannian spaces of constant zero curvature, constant positive curvature or Riemannian spaces admitting Killing vector fields of constant Riemannian length. 
\end{abstract}
\maketitle
\footnotetext[1]{Cs. Vincze is supported by the EFOP-3.6.1-16-2016-00022 project. The project is co-financed by the European Union and the European Social Fund.}
\section*{Introduction}

The notion of generalized Berwald manifolds goes back to V. Wagner \cite{Wag1}. They are Finsler manifolds admitting linear connections such that the parallel transports preserve the Finslerian length of tangent vectors (compatibility condition). The basic questions of the theory are the unicity of the compatible linear connection and its expression in terms of the canonical data of the Finsler manifold (intrinsic characterization). In case of a classical Berwald manifold admitting a compatible linear connection with zero torsion, the intrinsic characterization is the vanishing of the mixed curvature tensor of the canonical horizontal distribution. In general the intrinsic characterization of the compatible linear connection is based on the so-called \emph{averaged Riemannian metric}  given by integration of the Riemann-Finsler metric on the indicatrix hypersurfaces. By the fundamental result of the theory \cite{V5} such a linear connection must be metrical with respect to the averaged Riemannian metric. Therefore the linear connection is uniquely determined by its torsion tensor. Following Agricola-Friedrich \cite{AF} consider the decomposition
$$T(X,Y):=T_1(X,Y)+T_2(X,Y),\ \ \textrm{where}\ \ \displaystyle{T_1(X,Y):=T(X,Y)-\frac{1}{n-1}\big(\tilde{T}(X)Y-\tilde{T}(Y)X\big)},$$
$\tilde{T}$ is the trace tensor of the torsion and
\begin{equation}
\label{trace}
T_2(X,Y):=\frac{1}{n-1}\left(\tilde{T}(X)Y-\tilde{T}(Y)X\right).
\end{equation}
In case of 2D the torsion tensor is automatically of the form (\ref{trace}); see \cite{V13}. If the dimension is at least three then the trace-less part can be divided into two further components
$$T_1(X,Y)=A_1(X,Y)+S_1(X,Y)\ \ \Rightarrow\ \ T(X,Y)=A_1(X,Y)+S_1(X,Y)+T_2(X,Y)$$
by separating the totally anti-symmetric/axial part $A_1$. Therefore
we have eight possible classes of generalized Berwald manifolds depending on the surviving terms such as classical Berwald manifolds admitting torsion-free compatible linear connections \cite{Szab1} (we have no surviving terms) or Finsler manifolds admitting semi-symmetric compatible linear connections (we have no trace-less part) \cite{V6}, \cite{V9}, \cite{V10} and \cite{V11}. In the paper we are going to give explicit formulas for the linear connections with totally anti-symmetric torsion preserving the Finslerian length of tangent vectors in case of three-dimensional Finsler manifolds. The results are based on averaging of (intrinsic) Finslerian quantities by integration over the indicatrix surfaces. They imply some consequences for the base manifold as a Riemannian space with respect to the averaged Riemannian metric. The possible cases are Riemannian space forms of constant zero curvature, constant positive curvature or Riemannian spaces admitting Killing vector fields of constant Riemannian length. 

\section{Notations and terminology}

Let $M$ be a differentiable manifold with local coordinates $u^1, \ldots, u^n.$ The induced coordinate system of the tangent manifold $TM$ consists of the functions $x^1, \ldots, x^n$ and $y^1, \ldots, y^n$. For any $v\in T_pM$, $x^i(v):=u^i\circ \pi (v)=p$ and $y^i(v)=v(u^i)$, where $i=1, \ldots, n$ and $\pi \colon TM \to M$ is the canonical projection. 
\subsection{Finsler metrics} A Finsler metric is a continuous function $F\colon TM\to \mathbb{R}$ satisfying the following conditions:
\begin{itemize}
\item[(F1)] $\displaystyle{F}$ is smooth on the complement of the zero section (regularity),
\item[(F2)] $\displaystyle{F(tv)=tF(v)}$ for all $\displaystyle{t> 0}$ (positive homogenity),
\item[(F3)] the Hessian
$$g_{ij}=\frac{\partial^2 E}{\partial y^i \partial y^j}, \ \ \textrm{where} \ \ E=\frac{1}{2}F^2$$
is positive definite at all nonzero elements $\displaystyle{v\in T_pM}$ (strong convexity).
\end{itemize}

The so-called \emph{Riemann-Finsler metric} $g$ is constituted by the components $g_{ij}$. It is defined on the complement of the zero section because the second order partial differentiability of the energy function at the origin does not follow automatically: if $E$ is of class $C^2$ on the entire tangent manifold $TM$  then, by the positively homogenity of degree two, it follows that $E$ is quadratic on the tangent spaces, i.e. the space is Riemannian.
The Riemann-Finsler metric makes each tangent space (except at the origin) a Riemannian manifold with standard canonical objects such as the {\emph {volume form}} $\displaystyle{d\mu=\sqrt{\det g_{ij}}\ dy^1\wedge \ldots \wedge dy^n}$,
the \emph {Liouville vector field} $\displaystyle{C:=y^1\partial /\partial y^1 +\ldots +y^n\partial / \partial y^n}$ and the {\emph {induced volume form}}
$$\mu=\sqrt{\det g_{ij}}\ \sum_{i=1}^n (-1)^{i-1} \frac{y^i}{F} dy^1\wedge\ldots\wedge dy^{i-1}\wedge dy^{i+1}\ldots \wedge dy^n.$$
 on the indicatrix hypersurface $\displaystyle{\partial K_p:=F^{-1}(1)\cap T_pM\ \  (p\in M)}$. 

\subsection{Generalized Berwald manifolds}

\begin{Def} A linear connection $\nabla$ on the base manifold $M$ is called \emph{compatible} to the Finslerian metric if the parallel transports with respect to $\nabla$ preserve the Finslerian length of tangent vectors. Finsler manifolds admitting compatible linear connections are called generalized Berwald manifolds.
\end{Def}

Suppose that the parallel transports with respect to $\nabla$ (a linear connection on the base manifold) preserve the Finslerian length of tangent vectors and let $X_t$ be a parallel vector field along the curve $c\colon [0,1]\to M$. We have that
\begin{equation}
\label{paralleldiff}
(x^k\circ X_t)'={c^k}'\ \ \textrm{and}\ \ (y^k \circ X_t)'={X_t^k}'=-{c^i}'  X_t^j  \Gamma_{ij}^k\circ c
\end{equation}
because of the differential equation for parallel vector fields. If $F$ is the Finslerian fundamental function then
\begin{equation}
\label{eq:5}
(F \circ X_t)'=(x^k\circ X_t)'{\frac{\partial F}{\partial x^k}}\circ X_t+(y^k \circ X_t)'{\frac{\partial F}{\partial y^k}}\circ X_t
\end{equation}
and, by formula (\ref{paralleldiff}),
\begin{equation}
\label{eq:55}
(F \circ X_t)'={c^i}'\bigg(\frac{\partial F}{\partial x^i}-y^j {\Gamma}_{ij}^{k}\circ \pi \frac{\partial F}{\partial y^k}\bigg)\circ X_t.
\end{equation}
This means that the parallel transports with respect to $\nabla$ preserve the Finslerian length of tangent vectors (compatibility condition) if and only if
\begin{equation}
\label{cond1}
\frac{\partial F}{\partial x^i}-y^j {\Gamma}^k_{ij}\circ \pi \frac{\partial F}{\partial y^k}=0\ \  (i=1, \ldots,n),
\end{equation}
where the vector fields of type
\begin{equation}
\label{eq:6}
X_i^{h^{\nabla}}:=\frac{\partial}{\partial x^i}-y^j {\Gamma}^k_{ij}\circ \pi \frac{\partial}{\partial y^k}
\end{equation}
span the associated horizontal distribution belonging to $\nabla$.

\begin{Thm}
\label{heritage} \emph{\cite{V5}} If a linear connection on the base manifold is compatible to the Finslerian metric function then it must be metrical with respect to the averaged Riemannian metric
\begin{equation}
\label{averagemetric1}
\gamma_p (v,w):=\int_{\partial K_p} g(v, w)\, \mu=v^i w^j \int_{\partial K_p} g_{ij}\, \mu \ \ (v, w\in T_p M, p\in U).
\end{equation}
\end{Thm}

\section{Three-dimensional Finsler manifolds admitting compatible linear connections with totally anti-symmetric torsion tensor}

Suppose that $\nabla$ is a compatible linear connection of a three-dimensional generalized Berwald manifold. By Theorem \ref{heritage}, such a linear connection must be metrical with respect to the averaged Riemannian metric (\ref{averagemetric1}) given by integration of the Riemann-Finsler metric on the indicatrix hypersurfaces.
Therefore $\nabla$ is uniquely determined by its torsion tensor. Taking vector fields with pairwise vanishing Lie brackets on the neighbourhood $U$ of the base manifold, the Christoffel process implies that
$$X\gamma(Y,Z)+Y\gamma(X,Z)-Z\gamma(X,Y)=$$
$$2\gamma(\nabla_X Y, Z)+\gamma(X, T(Y,Z))+\gamma(Y, T(X,Z))-\gamma(Z, T(X,Y))$$
and, consequently, 
\begin{equation}
\label{Cproc}
\gamma(\nabla^*_X Y,Z)=\gamma(\nabla_X Y, Z)+\frac{1}{2}\left(\gamma(X, T(Y,Z))+\gamma(Y, T(X,Z))-\gamma(Z, T(X,Y))\right),
\end{equation}
where $\nabla^*$ denotes the L\'{e}vi-Civita connection of the averaged Riemannian metric. 
\begin{Def} The torsion tensor is totally anti-symmetric if its lowered tensor
$$T_{\flat}(X,Y,Z):=\gamma (T(X,Y),Z)$$
belongs to $\wedge^3 M$.
\end{Def}
\begin{Cor} If $\nabla$ is a metrical linear connection with totally anti-symmetric torsion then
\begin{equation}
\label{totanti}
\nabla^*_X Y=\nabla_X Y-\frac{1}{2}T(X,Y)
\end{equation}
and the geodesics of $\ \nabla^*$ and $\nabla$ coincide. 
\end{Cor} 
If $\dim M=3$ then $\dim \wedge^3 M=1$ and, consequently, 
$$T_{\flat}\left(\frac{\partial}{\partial u^i}, \frac{\partial}{\partial u^j}, \frac{\partial}{\partial u^k}\right)=f \gamma \left(\frac{\partial}{\partial u^i} \times_{\gamma} \frac{\partial}{\partial u^j}, \frac{\partial}{\partial u^k}\right)=f\det \gamma_{ij}$$
for some local function $f\colon U\to \mathbb{R}$, where the orientation is choosen such that the coordinate vector fields represent a positive basis. This means that
\begin{equation}
\label{formulamain06}
\nabla^*_X Y=\nabla_X Y-\frac{f}{2}X \times_{\gamma} Y.
\end{equation}
Taking the Riemannian energy $\displaystyle{E^*(v):=\gamma(v,v)/2}$, the Riemann-Finsler metric is
$\displaystyle{g^*_{ij}=\gamma_{ij}\circ \pi}$ and the cross product of vertical vector fields is defined by 
$$g^* \left(\frac{\partial}{\partial y^i}\times_{g^*}\frac{\partial}{\partial y^j}, \frac{\partial}{\partial y^k}\right)=\det g^*_{ij}=\det \gamma_{ij}\circ \pi$$
with bilinear extension. Formula (\ref{formulamain06}) can be written in terms of the induced horizontal structures as follows. Since the horizontal distributions induced by $\nabla^*$ and $\nabla$ are spanned by the vector fields
$$X_i^{h^{*}}=\frac{\partial }{\partial x^i}-y^j {\Gamma^*}_{ij}^l\circ \pi\frac{\partial}{\partial y^l}\ \ \textrm{and}\ \ X_i^{h^{\nabla}}=\frac{\partial }{\partial x^i}-y^j {\Gamma}_{ij}^l\circ \pi\frac{\partial}{\partial y^l},$$
respectively, we have, by formula (\ref{formulamain06}), that
\begin{equation}
\label{comparison}
X_i^{h^*}=X_i^{h^{\nabla}}+f\circ \pi  V_i, \ \ \textrm{where}\ \ V_i=\frac{1}{2} \frac{\partial}{\partial y^i}\times_{g^*} C \ \ (i=1,2,3).
\end{equation}
\subsection{Three-dimensional Finsler manifolds admitting compatible linear connections with totally anti-symmetric torsion tensor} Let $M$ be a three-dimensional Finsler manifold admitting a compatible linear connection with totally anti-symmetric torsion tensor. U\-sing the comparison formula (\ref{comparison}), the compatibility condition (\ref{cond1}) gives that 
\begin{equation}
\label{key1}
X_i^{h^*}E=f\circ \pi V_i E \ \ (i=1, 2, 3) \ \ \Rightarrow \ \ VE=f\circ \pi \sum_{i=1}^3 (V_i E)^2,
\end{equation}
where the vector field $V$ is defined by the formula $\displaystyle{V:=\sum_{i=1}^3 (V_i E) X_i^{h^*}}$.

\begin{Lem}
\label{allorno}
If $\displaystyle{\sum_{i=1}^3 (V_i E)_v^2 = 0}$ for any $v\in T_pM$ then the Finslerian indicatrix $\partial K_p$ is a sphere with respect to the averaged Riemannian metric.
\end{Lem}

\begin{Pf}
Since the infinitesimal rotation represented by the matrix 
$$\left(\begin{matrix}
0 & -v^1 & \ \ v^2  \\
\ \ v^1 & 0 & -v^3\\
-v^2 & \ \ v^3 & 0
\end{matrix}\right)$$
is of rank $2$, $\displaystyle{\sum_{i=1}^3 (V_i E)_v^2 = 0}$ implies that the vector fields $V_1$, $V_2$ and $V_3$ span the tangent plane to the Finslerian indicatrix at any $v\in \partial K_p$. Therefore its Euclidean normal vector field (with respect to $g^{*}$) is proportional to $C$. Taking a curve $c\colon [0,1]\to \partial K_p$ we have 
$$0=g^*_{c(t)}\left(C\circ c(t), c'(t)\right)=\gamma_p(c(t), c'(t))=\frac{1}{2}\gamma(c,c)'(t),$$
i.e. the Euclidean norm of $c(t)$ is constant. Since the Finslerian indicatrix surface is arcwise connected this means that it is a sphere with respect to the averaged Riemannian metric.
\end{Pf}

\begin{Thm}
For a three-dimensional non-Riemannian Finsler manifold, the compatible linear connection with totally anti-symmetric torsion tensor must be of the form
\begin{equation}
\label{comparison01}
\nabla_X Y=\nabla^*_{X} Y+\frac{f}{2} X \times_{\gamma} Y,
\end{equation}
where $\nabla^*$ is the L\'{e}vi-Civita connection of the averaged Riemannian metric $\gamma$ and the function $f$ is given by
\begin{equation}
f(p)=\frac{1}{\sigma(p)}\int_{\partial K_p} VE\, \mu,
\end{equation}
where 
$$V=\sum_{i=1}^3 (V_i E) X_i^{h^*},\ \ \sigma(p)=\sum_{i=1}^3 \int_{\partial K_p} (V_i E)^2 \, \mu, \ \ V_i=\frac{1}{2} \frac{\partial}{\partial y^i}\times_{g^*} C. $$
\end{Thm}
\begin{Pf}
Since $\displaystyle{VE=f\circ \pi \sum_{i=1}^3 (V_i E)^2}$ it follows that
\begin{equation}
\label{bubu}
\int_{\partial K_p} VE \, \mu=f(p)\sum_{i=1}^3 \int_{\partial K_p} (V_i E)^2\, \mu
\end{equation}
If the integrand on the right hand side is zero then, by Lemma \ref{allorno}, $\partial K_p$ is a Euclidean sphere in $T_pM$ with respect to $\gamma$. In case of a generalized Berwald manifolds we have linear parallel transports between the tangent spaces. Since the translates of a quadratic surface are quadratic, this means that the manifold is Riemannian. Otherwise we can divide equation (\ref{bubu}) to express the function $f$.
\end{Pf}

\section{Curvature properties} 

Let a point $p\in M$ be fixed and consider the subgroup $G$ of orthogonal transformations with respect to the averaged inner product leaving the indicatrix $\partial K_p$ invariant in $T_pM$. Such a group is obviously closed in $O(3)$ and, consequently, it is compact. If we have a generalized Berwald manifold then the group $G$ is essentially independent of the choice of $p$ because the parallel translations with respect to the compatible linear connection $\nabla$ makes them isomorphic provided that the manifold is connected. On the other hand $G$  must be finite or reducible unless the manifold is Riemannian; see \cite[Remark 5]{V12}. According to Theorem \ref{heritage} it follows that $\textrm{Hol} \ \nabla\subset G$, i.e. the holonomy group of a compatible linear connection is finite or reducible in case of a a non-Riemannian generalized Berwald manifold. Using vector fields $X$, $Y$ and $Z$ with pairwise vanishing Lie-brackets on a neighbourhood $U\subset M$, the comparison formula (\ref{comparison01}) says that
\begin{equation}
\label{curv}
R(X,Y)Z=R^*(X,Y)Z+\frac{1}{2} \left((Xf)Y-(Yf)X\right)\times_{\gamma} Z+\frac{f^2}{4} \left(X\times_{\gamma} Y\right)\times_{\gamma} Z
\end{equation}
because of the Jacobi identity
$$\left(X\times_{\gamma} Y\right)\times_{\gamma} Z+\left(Z\times_{\gamma} X\right)\times_{\gamma} Y+\left(Y\times_{\gamma} Z\right)\times_{\gamma} X=0$$
and the product rule
$$\nabla_X^* (Y\times_\gamma Z)=\left(\nabla_X^* Y\right)\times_\gamma Z+Y\times_{\gamma} \left(\nabla_X^* Z\right).$$
Using the vector triple product extension formula
$$X\times_{\gamma} \left(Y\times_{\gamma} Z\right)=\gamma(X,Z)Y-\gamma(X,Y)Z$$
it follows that
\begin{equation}
\label{sect}
\gamma(R(X,Y)Y,X)=\gamma(R^*(X,Y)Y,X)-\frac{f^2}{4} \det \left(\begin{matrix}
\gamma(X,X) & \gamma(X,Y)\\
\gamma(X,Y) & \gamma(Y,Y)\end{matrix}\right).
\end{equation}

The curvature tensor of $\nabla$ obviously satisfies the curvature property
\begin{equation}
\label{curv1}
R(X,Y)Z=-R(Y,X)Z.
\end{equation}
Property
\begin{equation}
\label{curv2}
\gamma(R(X,Y)Z,W)=-\gamma(R(X,Y)W,Z)
\end{equation}
also holds because $\nabla$ is metrical with respect to the averaged Riemannian metric $\gamma$. We are going to investigate the Jacobi identity and the block symmetry.

\begin{Lem} 
\label{ccp}
The curvature tensor of $\nabla$ satisfies the Jacobi identity
\begin{equation}
\label{curv3}
R(X,Y)Z+R(Z,X)Y+R(Y,Z)X=0
\end{equation}
if and only if the function $f$ is constant.
\end{Lem}
\begin{Pf}
As a straightforward calculation shows 
\begin{equation}
\label{jacobi}
R(X,Y)Z+R(Z,X)Y+R(Y,Z)X\stackrel{(\ref{curv})}{=}(Xf) Y\times_\gamma Z+(Zf) X\times_\gamma Y+(Yf) Z\times_\gamma X.
\end{equation}
Taking the inner product of both sides with $Z$ (for example) it can be easily seen that the left hand side is zero if and only if $Zf=0$ for any vector field $Z$ on the base manifold. 
\end{Pf}

\begin{Rem} \emph{To complete the list of the classical curvature properties we need to investigate the so-called block-symmetry
\begin{equation}
\label{curv4}
\gamma(R(X,Y)Z,W)=\gamma(R(Z,W)X,Y);
\end{equation}
especially, if the dimension is $3$, then properties (\ref{curv3}) and (\ref{curv4}) are equivalent to each other for any curvature tensor satisfying (\ref{curv2}). It follows by a pure algebraic way \cite[Remark 3.11]{V11}.}
\end{Rem}

\subsection{The case of finite holonomy group} Suppose that $G$ is finite and, consequently, the holonomy group of the compatible linear connection is also finite, i.e. its curvature is zero. 

\begin{Thm} If $M$ is a connected three-dimensional non-Riemannian Finsler manifold admitting a compatible flat linear connection with totally anti-symmetric torsion tensor then 
\begin{itemize}
\item $M$ is a classical Berwald manifold of constant zero sectional curvature with respect to the averaged Riemannian metric or
\item $M$ is a proper generalized Berwald manifold of constant positive sectional curvature with respect to the averaged Riemannian metric.
\end{itemize}
\end{Thm}

\begin{Pf}
If we have a compatible linear connection with zero curvature then the classical curvature properties (\ref{curv1}) - (\ref{curv4}) are all satisfied and $f$ must be a constant function due to Lemma \ref{ccp}. The result comes from the comparison formula (\ref{sect}): if the (constant) function $f$ is identically zero then we have a Riemannian space form of constant zero curvature. Otherwise it is of constant positive curvature. 
\end{Pf}

\begin{Rem}
{\emph{If $M$ is complete then, by the Killing-Hopf theorem of Riemannian geometry, it follows that the universal cover of $M$ (as a Riemannian space with respect to the averaged Riemannian metric) is isometric to $\mathbb{R}^3$ or the Euclidean unit sphere $S^3\subset \mathbb{R}^4$. Otherwise the manifold (as a non-Riemannian Finsler space) does not admit a compatible flat linear connection with totally anti-symmetric torsion tensor.}}
\end{Rem}

\subsection{The case of non-finite reducible holonomy group} 

\begin{Thm}
\label{thm:4} If $M$ is a connected three-dimensional non-Riemannian Finsler manifold admitting a compatible non-flat linear connection $\nabla$ with totally anti-symmetric torsion tensor then there exists a one-dimensional distribution $\mathcal{D}$ such that
\begin{itemize}
\item any local section of constant length is a covariant constant vector field with respect to $\nabla$,
\item any local section of constant length is a Killing vector field of constant length with respect to the averaged Riemannian metric. 
\end{itemize}
\end{Thm}

\begin{proof} Suppose that $R_p(X,Y)\neq 0$ for some point $p\in M$, i.e. the holonomy group of the compatible linear connection contains a one-parameter subgroup of orthogonal transformations with respect to the averaged Riemannian metric leaving the indicatrix $\partial K_p$ invariant. Since $M$ is a non-Riemannian generalized Berwald manifold it follows that $G$ containing all orthogonal transfromations leaving the indicatrix $\partial K_p$ invariant can not be transitive or finite because $G\supset \textrm{Hol} \ \nabla$. Therefore it is a reducible group containing a one-parameter subgroup of rotations; see \cite[Remark 5]{V12}. The distribution is constituted by the one-dimensional invariant subspace of $G$ as the point of the manifold is varying. The local sections of constant length must be constructed as follows: if $\beta_p$ generates the one-dimensional invariant subspace (the common axis of the rotations) of $G$ in $T_pM$ then, by parallel transports with respect to $\nabla$, we can extend it to a local covariant constant section $\beta$, i.e. $\displaystyle{\nabla_{X_q} \beta =0}$ for any vector field $X$ on $M$ and $q\in U$, where $U$ is a local neighbourhood around $p\in M$. Since $\displaystyle{\nabla_X \beta=0}$ it follows that
$$\nabla^*_X \beta\stackrel{(\ref{comparison01})}{=}-\frac{f}{2} X \times_{\gamma} \beta$$
and, consequently,
$$\left(\mathcal{L}_{\beta} \gamma\right)(X,Y)=\beta \gamma(X,Y)-\gamma([\beta, X],Y)-\gamma([\beta,Y],X),$$
where
$$[\beta,X]=\nabla^*_{\beta}X-\nabla^*_X \beta=\nabla^*_{\beta}X+\frac{f}{2} X \times_{\gamma} \beta=\nabla^*_{\beta}X -\frac{f}{2} \beta \times_{\gamma} X.$$
Therefore
$$\left(\mathcal{L}_{\beta} \gamma\right)(X,Y)=\beta \gamma(X,Y)-\gamma(\nabla^*_{\beta}X,Y)-\gamma(\nabla^*_{\beta}Y,X)+\frac{f}{2}\gamma (\beta \times_{\gamma} X,Y)+\frac{f}{2}\gamma (\beta \times_{\gamma} Y,X)=$$
$$\beta \gamma(X,Y)-\gamma(\nabla^*_{\beta}X,Y)-\gamma(\nabla^*_{\beta}Y,X)=0$$
as was to be proved. 
\end{proof}

The following example can be considered as a converse of the global version of Theorem \ref{thm:4}.

\subsection{An example} Suppose that $M$ is a connected Riemannian manifold admitting a Killing vector field $\beta$ of unit length. An easy direct computation shows that 
$$\left(\mathcal{L}_{\beta} \gamma\right)(X,Y)=\gamma(\nabla^*_X \beta,Y)+\gamma(\nabla^*_Y \beta ,X)=0,$$
i.e. the Hesse form $\gamma(\nabla^*_X \beta,Y)$ is anti-symmetric. On the other hand $\gamma(\nabla^*_X \beta,\beta)=0$ because $\beta$ is of constant length. Therefore
$$\gamma(\nabla^*_X \beta,Y)=-\frac{f}{2} \gamma(\beta \times_{\gamma} X, Y)$$
for some function $f$, i.e.
$$\nabla^*_X \beta=-\frac{f}{2} \beta \times_{\gamma} X.$$
Taking the metric connection $\nabla$ with torsion $T(X,Y)=f X\times_{\gamma} Y$ it follows, by the comparison formula (\ref{comparison01}), that $\beta$ is a covariant constant vecor field with respect to $\nabla$. This means that for any vector fields $X$ and $Y$
$$R(X,Y)\beta=0, \ \ \textrm{i.e.}  \ \ R(X,Y)Z=r(X,Y) \beta\times_{\gamma} Z$$
for some anti-symmetric scalar-valued form $r\in \wedge^2 M$  because $\beta$ is the vector invariant of the anti-symmetric mapping $\displaystyle{Z \mapsto R(X,Y)Z}$. By the Ambrose-Singer theorem, the unit component of the holonomy group (the so-called restricted holonomy group) of $\nabla$ is the one-parameter rotational group generated by $\beta$ at each point of the manifold. Taking a point $p\in M$ let us choose a non-qadratic convex revolution surface around the axis of $\beta_p$. For an explicite example consider a trifocal ellipsoid (it is a kind of generalized conics instead of the classical conics of the Riemannian geometry) body defined by the equation
\begin{equation}
\label{trifocal}
\|v+\beta_p\|+ \|v\|+\|v-\beta_p\|\leq \textrm{const.};
\end{equation}
the focal set consists of $\pm \beta_p$, ${\bf 0}$ and the constant is large enough to contain the focal points in the interior of the body to avoid singularities. Using parallel transports with respect to $\nabla$ we can extend (\ref{trifocal}) to each point of the manifold. Note that $\pm \beta_p$ in the focal set provide that (\ref{trifocal}) is invariant under not only the restricted holonomy group but the entire one including possibly reflections about the two-dimensional invariant subspace. Such a smoothly varying family of convex bodies induces a (non-Riemannian) fundamental function $F$ such that it is invariant under the parallel transport with respect to $\nabla$.

\begin{Rem}\emph{Killing vector fields of constant length naturally appear in different geometric constructions such as $K$-contact and Sasakian manifolds \cite{Blair}, \cite{BMS} and \cite{BG}. There are many restrictions to the existence of Killing vector fields of constant length on a Riemannian manifold; for a comprehensive survey see \cite{BN}: for example, if a compact Riemannian manifold admits such a vector field then its Euler characteristic must be zero in the sense of a Theorem due to H. Hopf \cite[Section 1]{BN}.}
\end{Rem}


\begin{thebibliography}{99}

\bibitem{AF} I. Agricola and T. Friedrich, \emph{On the holonomy of connections with skew-symmetric torsion}, Math. Ann. {\bf 328} (4) (2004), pp. 711-748.

\bibitem{Blair} D. Blair, \emph{Contact manifolds in Riemannian geometry}, Springer Lectures Notes in Math., V. {\bf 509}, Springer Verlag, Berlin and New York, 1976. 

\bibitem{BMS} F. Belgun, A. Moroianu, and U. Semmelmann, \emph{Symmetries of contact metric manifolds}, Geom. Dedic.
{\bf 101} (1) (2003), pp. 203-216.

\bibitem{BN} V.N. Beretovskii and Yu. G. Nikonorov, \emph{Killing vector fields of constant length
on Riemannian manifolds}, Siberian Mathematical Journal {\bf 49} (3) (2008), pp 395-407.

\bibitem{BG} C. Boyer and K. Galicki, \emph{On Sasakian-Einstein geometry}, Internat. J. Math. {\bf 11} (7) (2000), pp. 873-909.

%\bibitem{A1} G. S. Asanov, \emph{Finslerian metric functions over the product $\mathbb{R}\times M$ and their potential appliacations}, Rep. on Math. Phys., Vol. 41, No. 1 (1998), 117-132.

%\bibitem{A2} G. S. Asanov, Finsleroid-Finsler spaces of positive definite and relativistic types, Rep. Math. Phys. 58 (2006), pp. 275-300.

%\bibitem{B} D. Bao, On two curvature-driven problems in Riemann-Finsler geometry, Advanced Studies in Pure Mathematics 48, 2007, pp. 19-71.

%\bibitem{BSC} D. Bao, S. - S. Chern and Z. Shen, \emph{An Introduction to Riemann-Finsler geometry}, Springer-Verlag, 2000.

%\bibitem{Berwald1} L. Berwald, \emph{\"{U}ber zweidimensionale allgemeine metrische R\"{a}ume}, J. reine angew. Math. 156 (1927), 191-210 and 211-222.

%\bibitem{Berwald2} L. Berwald, \emph{On Finsler and Cartan Geometries III, Two-dimensional Finsler spaces with rectilinear extremals}, Ann. of Math. 42 (1941), 84-112.

%\bibitem{G1} J. Grifone, \emph{Structure presque-tangente et connexions I}, Ann. Inst. Fourier, Grenoble 22 (1) (1972),287-334.

%\bibitem{G2} J. Grifone, \emph{Structure presque-tangente et connexions II}, Ann. Inst. Fourier, Grenoble 22 (3) (1972), 291-338.


%\bibitem{H1}
%M. Hashiguchi, \emph{On Wagner's generalized Berwald space}, J. Korean Math. Soc. 12 (1) (1975), 51-61.

%\bibitem{H2} M. Hashiguchi, \emph{On conformal transformations of Finsler metrics}, J. Math. Kyoto Univ. 16 (1976), 25-50.

%\bibitem{HY} M. Hashiguchi and Y. Ichijy$\bar{\textrm{o}}$, \emph{On conformal transformations of Wagner spaces}, Rep. Fac. Sci. Kagoshima Univ. (Math., Phys., Chem.) No. 10 (1977), 19-25.

%\bibitem{M1} M. Matsumoto, \emph {Foundations of Finsler Geometry and Special Finsler spaces}, Kaisheisa Press, Otsu (1986).

%\bibitem{M2} M. Matsumoto, \emph {Conformally Berwald and conformally flat Finsler spaces}, Publ. Math. Debrecen, 58 (1-2) (2001), 275-285.

\bibitem{Szab1} Z. I. Szab\'{o}, \emph{Positive definite Berwald spaces. Structure theorems on Berwald spaces}, Tensor (N. S.) {\bf 35} (1) (1981), pp. 25-39.

%\bibitem{VV} Sz. Vattam\'{a}ny and Cs. Vincze, \emph{On a new geometrical derivation of two-dimensional Finsler manifolds with constant main scalar}, Period. Math. Hungar. 48 (1-2) (2004), 61-67.

%\bibitem{VV1} Sz. Vattam\'{a}ny and Cs. Vincze, \emph{Two-dimensional Landsberg manifolds with vanishing Douglas tensor}, Annales Univ. Sci. Budapest., 44 (2001), 11-26.

%\bibitem{V1} Cs. Vincze, \emph{An intrinsic version of Hashiguchi-Ichijy$\bar{\textrm{o}}$'s theorems for Wagner manifolds}, SUT J. Math. 35 (2) (1999), 263-270.

%\bibitem{V2} Cs. Vincze, \emph{On Wagner connections and Wagner manifolds}, Acta Math. Hung. 89 (1-2) (2000), 111-133.

%\bibitem{V3} Cs. Vincze, \emph{On conformal equivalence of Berwald manifolds all of whose indicatrices have positive curvature}, SUT J. Math. 39 (1) (2003), 15-40.

%\bibitem{V4} Cs. Vincze, \emph{On the curvature of the indicatrix surface in three-dimen\-si\-onal Minkowski spaces}, Period. Math. Hungar. Vol 48 (1-2), 2004, 69-76.

\bibitem{V5}
Cs. Vincze, \emph{A new proof of Szab\'{o}' s theorem on the Riemann-metrizability of Berwald manifolds}, J. AMAPN,
21 (2005), 199-204.

\bibitem{V6}
Cs. Vincze, \emph{On a scale function for testing the conformality of Finsler manifolds
to a Berwald manifold}, Journal of Geometry and Physics. 54 (2005), 454-475.

%\bibitem{V7} Cs. Vincze, \emph{On geometric vector fields of Minkowski spaces and their applications}, J. Diff. Geom. and Its Appl. 24 (2006), 1-20.

%\bibitem{V8}
%Cs. Vincze, \emph{On an existence theorem of Wagner manifolds}, Indag. Mathem., N.S.,
%17 (1) (2006),
%129-145.

\bibitem{V9}
Cs. Vincze, \emph{On Berwald and Wagner manifolds}, J. AMAPN,
24 (2008) 169-178.

\bibitem{V10} Cs. Vincze, \emph{Generalized Berwald manifolds with semi-symmetric linear connections}, Publ. Math. Debrecen 83 (4) (2013), pp. 741-755.

\bibitem{V11} Cs. Vincze, \emph{On a special type of generalized Berwald manifolds: semi-symmetric linear connections preserving the Finslerian length of tangent vectors}, "Finsler geometry: New methods and Perspectives", European Journal of Mathematics, December 2017, Volume 3, Issue 4, pp 1098 - 1171 .

\bibitem{V12} Cs. Vincze, \emph{Lazy orbits: an optimization problem on the sphere}, Journal of Geometry and Physics
Volume 124, January 2018, Pages 180-198. 

\bibitem{V13} Cs. Vincze, T. Khosdhani, Z. Mehdizadeh and M. Ol\'{a}h, \emph{Intrinsic characterizations of compatible linear connections of two-dimensional generalized Berwald manifolds}, submitted for piblication to J. of Diff. Geom. and Its Appl. 

\bibitem{Wag1}
V. Wagner, \emph{On generalized Berwald spaces}, CR Dokl. Acad. Sci. USSR (N.S.)
{\bf 39} (1943), pp. 3-5.

\end{thebibliography}
\end{document}